\documentclass[a4paper,11pt,reqno,smallextended]{amsart} 

\usepackage{latexsym}
\usepackage{amssymb,amsmath,amsopn}
\usepackage{graphics} 

\usepackage{url}

\newtheorem{theorem}{Theorem}[section]
\newtheorem{lemma}[theorem]{Lemma}

\newtheorem{corollary}[theorem]{Corollary}

\newtheorem{proposition}[theorem]{Proposition}

\theoremstyle{definition}
\newtheorem{definition}[theorem]{Definition}

\newtheorem*{example}{Example}

\newcommand{\scup}{\hskip0.5pt\cup\hskip0.5pt}

\DeclareMathOperator{\GL}{GL}
\DeclareMathOperator{\SL}{SL}
\DeclareMathOperator{\sgn}{sgn}

\DeclareMathOperator{\Sym}{Sym}
\DeclareMathOperator{\Inf}{Inf}

\renewcommand{\theta}{\vartheta}

\newcommand{\N}{\mathbb{N}}

\newcommand{\Q}{\mathbb{Q}}

\newcommand{\Ind}{\big\uparrow}
\newcommand{\Res}{\big\downarrow}
\newcommand{\ind}{\!\!\uparrow}

\newcommand{\oa}{1_{\raisebox{0pt}{$\scriptscriptstyle 1$}}}
\newcommand{\ob}{1_{\raisebox{0pt}{$\scriptscriptstyle 2$}}}
\newcommand{\oc}{1_{\raisebox{0pt}{$\scriptscriptstyle 3$}}}
\newcommand{\od}{1_{\raisebox{0pt}{$\scriptscriptstyle 4$}}}

\newcommand{\pa}{2_{\raisebox{0pt}{$\scriptscriptstyle 1$}}}
\newcommand{\pb}{2_{\raisebox{0pt}{$\scriptscriptstyle 2$}}}

\newcommand{\qa}{3_{\raisebox{0pt}{$\scriptscriptstyle 1$}}}

\newcommand{\ra}{4_{\raisebox{0pt}{$\scriptscriptstyle 1$}}}

\newcommand{\multibinom}[2]{\left(\!\middle(\genfrac{}{}{0pt}{}{#1}{#2}\middle)\!\right)}

\newcounter{thmlistcnt}
\newenvironment{thmlist}%
	{\setcounter{thmlistcnt}{0}%
	\begin{list}{\emph{(\roman{thmlistcnt})}}{%
		\usecounter{thmlistcnt}%
		\setlength{\topsep}{0pt}%
		\setlength{\leftmargin}{17pt}%
		\setlength{\itemsep}{0pt}%
		\setlength{\labelwidth}{12pt}
		\setlength{\itemindent}{0pt}}%
	}%
	{\end{list}}%
	
\newcounter{defnlistcnt}
\newenvironment{defnlist}%
	{\setcounter{defnlistcnt}{0}%
	\begin{list}{(\roman{defnlistcnt})}{%
		\usecounter{defnlistcnt}%
		\setlength{\topsep}{0pt}%
		\setlength{\leftmargin}{18.5pt}%
		\setlength{\itemsep}{0pt}%
		\setlength{\labelwidth}{18.5pt}
		\setlength{\itemindent}{0pt}}%
	}%
	{\end{list}}%

\let\oldotimes\otimes
\renewcommand{\otimes}{{}\oldotimes{}} 
\newcommand{\cotimes}{\!\otimes}       


\begin{document}

\thispagestyle{empty}

\title[Constituents of Foulkes characters]{Minimal and maximal constituents of twisted Foulkes characters}

\author{Rowena Paget and Mark Wildon}

\date{September 2014}
\maketitle

\begin{abstract}
We prove combinatorial rules that give the minimal and maximal partitions labelling the irreducible constituents
of a family of characters for the symmetric group that generalize Foulkes permutation characters. 
Restated in the language of symmetric functions, our results determine all minimal and maximal
partitions that label Schur functions appearing in the plethysms $s_\nu \circ s_{(m)}$. As a corollary we prove
two conjectures of Agaoka on the lexicographically least 
constituents of the plethysms $s_\nu \circ s_{(m)}$ and
$s_\nu \circ s_{(1^m)}$.
\end{abstract}

\thispagestyle{empty}

\section{Introduction}\label{sec:intro}

Fix $m$, $n \in \N$ and let $S_m \wr S_n \le S_{mn}$ be the transitive imprimitive wreath product of the 
symmetric groups $S_m$ and $S_n$. 
The \emph{Foulkes character} $\phi^{(m^n)}$ is the permutation character arising from the action of 
$S_{mn}$ 
on the cosets of \hbox{$S_m \wr S_n$}. 
Finding the decomposition of $\phi^{(m^n)}$ into irreducible characters of $S_{mn}$ is a long-standing open
problem that spans representation theory and algebraic combinatorics; a solution to this problem
would also solve Foulkes' Conjecture (see \cite[end~\S 1]{Foulkes}). 
Equivalently, one may ask for the decomposition of $\Sym^n(\Sym^m \! E)$ into irreducible $\GL(E)$-modules, where~$E$ is a finite-dimensional 
rational vector space, or, taking formal characters, for the decomposition of
the plethysm $s_{(n)} \circ s_{(m)}$ as an integral linear combination of Schur functions. 
The problem of
finding a clearly positive combinatorial rule for these coefficients was identified by Stanley in 
Problem 9 of~\cite{StanleyPositivity}
as one of the key open problems in algebraic combinatorics. 
We survey the existing results in Section~\ref{sec:background} below.

In this paper we study a generalization of Foulkes characters.
Let $\nu$ be a partition of $n$. 
Let $\Inf_{S_n}^{S_m \wr S_n} \chi^\nu$ denote
the character of $S_m \wr S_n$ inflated from the irreducible character of $\chi^\nu$
of $S_n$ using the canonical quotient map
$S_m \wr S_n \rightarrow S_n$. Let 
\[ \phi^{(m^n)}_\nu = 
(\Inf_{S_n}^{S_m \wr S_n} \!\chi^\nu)\Ind_{S_m \wr S_n}^{S_{mn}}.\]
We call these characters \emph{twisted
Foulkes characters}. The corresponding polynomial representation of $\mathrm{GL}(E)$ is
$\mathrm{Sym}^\nu\bigl( \mathrm{Sym}^m E \bigr)$, and the corresponding plethysm is $s_\nu \circ s_{(m)}$.

The two main results of this paper give combinatorial rules that determine the minimal partitions and the maximal partitions in the dominance order that label the irreducible constituents of these characters. 
As a corollary, we prove two conjectures of Agaoka \cite{Agaoka} on the lexicographically least
constituents
of the plethysms $s_\nu \circ s_{(m)}$ and $s_\nu \circ s_{(1^m)}$.

To state our main results we need the following  definitions. Let $\lambda'$ denote
the conjugate partition to a partition $\lambda$ and let $\unlhd$ denote the dominance order
on partitions.

\begin{definition}{\ }
\begin{defnlist}
\item  
 A \emph{set family} $\mathcal{P}$ of \emph{shape} $(m^n)$ is a collection of $n$ distinct $m$-subsets of~$\N$. The \emph{type} of the set family $\mathcal{P}$, if defined, is the partition $\lambda$ such
 that the number of sets in $\mathcal{P}$ that contain $i$ is $\lambda_i'$.

\item 
Let $\mathcal{P}_1, \ldots \mathcal{P}_c$ be set families. Then $(\mathcal{P}_1,\ldots,\mathcal{P}_c)$ is called a \emph{set family tuple}. The \emph{type} of the  set family tuple $(\mathcal{P}_1,\ldots,\mathcal{P}_c)$,  if defined, is the partition $\lambda$ 
such that the total number of sets in the set families $\mathcal{P}_1, \ldots, \mathcal{P}_c$
that contain $i$ is $\lambda_i'$.
\end{defnlist}
\end{definition}

Not all set family tuples possess a type, but we shall be primarily concerned with those that do. 
A set family $\mathcal{P}$ of type $\lambda$ is \emph{minimal} if there is no set family $\mathcal{R}$ 
of type $\mu \lhd \lambda$ that has the same shape as $\mathcal{P}$. 
A set family tuple $(\mathcal{P}_1,\ldots,\mathcal{P}_c)$ of type $\lambda$ is called \emph{minimal} if there is no set family tuple $(\mathcal{R}_1,\ldots,\mathcal{R}_c)$ 
of type $\mu \lhd \lambda$ such that 
each $\mathcal{R}_i$ has the same shape as~$\mathcal{P}_i$.  

We now make a similar definition replacing sets by multisets.

\begin{definition}{\ }
\begin{defnlist}
\item 
 A \emph{multiset family} $\mathcal{Q}$ of \emph{shape} $(m^n)$ is a collection of $n$ distinct multisets each of cardinality $m$ 
 having elements in $\N$. The \emph{type} of the multiset family $\mathcal{Q}$, if defined, is the partition $\lambda$ such
 that $\lambda_i'$ is the total number of occurrences of $i$ in the multisets contained in     
 $\mathcal{Q}$.
\item 
Let $\mathcal{Q}_1, \ldots \mathcal{Q}_c$ be multiset families. Then $(\mathcal{Q}_1,\ldots,\mathcal{Q}_c)$ is called a \emph{multiset family tuple}. The \emph{type} of the multiset family tuple  $(\mathcal{Q}_1,\ldots,\mathcal{Q}_c)$,  if defined, is the partition $\lambda$ 
such that $\lambda_i'$ is the total number of occurrences of $i$ in the multisets contained in     
 $\mathcal{Q}_1, \ldots, \mathcal{Q}_c$.
\end{defnlist}
\end{definition}
Minimal multiset family tuples are then defined in the same way as  minimal set family tuples.

Given a character $\psi$ of $S_{r}$ and a partition $\lambda$ of $r \in \N$, we say
that $\chi^\lambda$ is a \emph{minimal constituent} of $\psi$ if $\langle \psi,
\chi^\lambda\rangle \ge 1$, and $\lambda$ is minimal in the dominance order on partitions of $r$
with this property. The definition of \emph{maximal constituent} is analogous.

Our two main results are as follows.

\begin{theorem}\label{thm:min}
Let $\nu$ be a partition of $n$ and let $\lambda$ be a partition of $mn$.  Set $\kappa = \nu$ if $m$ is even and $\kappa = \nu'$ 
if $m$ is odd. Let $k$ be the first part of~$\kappa$.
Then $\chi^\lambda$ is a minimal constituent
of $\phi_\nu^{(m^n)}$ if and only if there is a minimal set family tuple $(\mathcal{P}_1, \ldots, \mathcal{P}_k)$ of type $\lambda$ such that
each $\mathcal{P}_j$ has shape~$(m^{\kappa_j'})$.
\end{theorem}

\begin{theorem}\label{thm:max}
Let $\nu$ be a partition of $n$ with first part $\ell$ and let $\lambda$ be a partition of $mn$.  
Then $\chi^\lambda$ is a maximal constituent
of $\phi_\nu^{(m^n)}$ if and only if there is a minimal multiset family tuple
$(\mathcal{Q}_1, \ldots, \mathcal{Q}_\ell)$ of type $\lambda'$ such that
each $\mathcal{Q}_j$ has shape $(m^{\nu_j'})$.
\end{theorem}

We pause to give a small example of these theorems.
This example is continued in Section~\ref{subsec:ex}.

\begin{example}
By Theorem~\ref{thm:min} the minimal constituents of $\phi^{(2^4)}_{(2,1,1)}$ are $\chi^{(4,2,1,1)}$ and $\chi^{(3,3,2)}$, 
corresponding to the minimal set family tuples
\[ \bigl( \bigl\{ \{1,2\}, \{1,3\}, \{1,4\} \bigr\}, \bigl\{\{1,2\}\bigr\} \bigr) \text{ and}\ 
\bigl( \bigl\{ \{1,2\}, \{1,3\}, \{2,3\} \bigr\}, \bigl\{\{1,2\}\bigr\} \bigr), \]
respectively. By Theorem~\ref{thm:max} the maximal constituents of $\phi^{(2^4)}_{(2,1,1)}$ are $\chi^{(6,1,1)}$ and $\chi^{(5,3)}$,
corresponding to the minimal multiset family tuples
\[ \bigl( \bigl\{ \{1,1\}, \{1,2\}, \{1,3\} \bigr\}, \bigl\{\{1,1\}\bigr\} \bigr) \text{ and}\
\bigl( \bigl\{ \{1,1\}, \{1,2\}, \{2,2\} \bigr\}, \bigl\{\{1,1\}\bigr\} \bigr), \]
respectively. 
\end{example}

To prove Theorem~\ref{thm:min} we construct an explicit module affording the character $\phi_\nu^{(m^n)}$, using
the plethysm functor from  representations of $S_{n}$ to representations of $S_{mn}$
defined in Section~\ref{subsec:P} below. 
We then define explicit
homomorphisms from Specht modules into this module. These constructions are of independent interest.
In Section~\ref{subsec:further} we show that
our homomorphisms give irreducible characters appearing in $\phi^{(m^n)}_\nu$ beyond
those predicted by our two main theorems. 

The maximal constituents of $\phi_\nu^{(m^n)}$ are in bijection with the minimal constituents of 
$\sgn_{S_{mn}} \!\times\, \phi_\nu^{(m^n)}$. To prove Theorem~\ref{thm:max} we define explicit modules
affording these characters and determine their 
minimal constituents by adapting the arguments used to prove Theorem~\ref{thm:min}. 

The outline of this paper is as follows. The common preliminary results we need are collected in Sections~\ref{sec:plethysm}
and~\ref{sec:closed}. We give a complete proof of Theorem~\ref{thm:min} when $m$ is even in Section~\ref{sec:evenProof},
and indicate in Section~\ref{sec:oddProof} the modifications required for odd $m$. By contrast, it is possible
to prove both cases of Theorem~\ref{thm:max} in an almost uniform way: we do this in Section~\ref{sec:psi}.
We end in Section~\ref{sec:corollaries} with a number of corollaries of the main theorems. In particular, we prove the two conjectures of Agaoka mentioned above by determining the lexicographically
least and greatest constituents of the characters~$\phi_\nu^{(m^n)}$. 
We also give a necessary and
sufficient condition for $\phi_\nu^{(m^n)}$ to have a unique
minimal or maximal constituent, and find an $\SL(E)$-invariant subspace
in the polynomial representation corresponding to certain twisted Foulkes characters. 
Finally we  
show that $\phi^{(2^n)}_{(1^n)}$ has the interesting property that all its constituents are both minimal
and maximal; we use
our two main theorems to give a new proof of the decomposition of this character into irreducible 
characters. 

We remark that
Theorem~2.6 in the authors' earlier paper \cite{PagetWildon} 
is the special case of Theorem~\ref{thm:min} when $m$ is odd and $\nu = (n)$. The authors recently learned of
work by Klivans and Reiner \cite[Proposition~5.10]{KlivansReiner} which gives a result equivalent to this special case.
The proofs in this paper use some similar ideas to \cite{PagetWildon}, 
but are considerably shorter, and give  more general results.

\section{Background on plethysms}\label{sec:background}

Let $\nu$ be a partition of $n$. 
Under the characteristic isomorphism $\phi^{(m^n)}_{\nu}$ is sent to the
plethysm of Schur functions \hbox{$s_\nu \circ s_{(m)}$}
(see \cite[I, Appendix A, (6.2)]{Macdonald}). 
The existing results on the characters $\phi^{(m^n)}_\nu$ are limited
 and have
mainly been obtained using the methods of symmetric polynomials.
We shall use this language throughout this section. 
The following plethysms of the form $s_\nu \circ s_{(m)}$ have a known 
decomposition into Schur functions:
\begin{itemize}
\item[(i)] $s_{(1^2)} \circ s_{(m)}$, $s_{(2)} \circ s_{(m)}$,
$s_{(n)} \circ s_{(1^2)}$ and $s_{(n)} \circ s_{(2)}$; see Littlewood \cite{LittlewoodCharacters},
\item[(ii)] $s_{(3)} \circ s_{(m)}$; see Thrall \cite[Theorem~5]{Thrall} or Dent and Siemons
\cite[Theorem~4.1]{DentSiemons},
\item[(iii)] $s_\nu \circ s_{(m)}$ when $\nu$ is a partition of
$4$;
see Theorem 27 of Foulkes \cite{FoulkesPlethysm} for an explicit decomposition in a special case and the remarks on the general case immediately following, 
\item[(iv)] $s_\nu \circ s_{(m)}$ when $\nu$ is a partition of $2$, $3$ or $4$;
see Howe \cite[Section 3.5 and Remark 3.6(b)]{Howe}. Howe's statements are usually more convenient than Foulkes'.
\end{itemize}

There are several further results which, like our two main theorems, give information 
 about constituents of a
special form.  The Cayley--Sylvester formula states that the multiplicity of $s_{(mn-r,r)}$ in
$s_{(n)} \circ s_{(m)}$ 
is equal to the number of partitions of $r$ whose Young diagram is contained
in the Young diagram of $(m^n)$.
A striking generalization due to Manivel \cite{ManivelCayleySylvester} 
states that the two-variable symmetric function $(s_{(n^k)} \circ s_{(m+k-1)})(x_1,x_2)$ 
is symmetric under any permutation of $m$, $n$ and $k$.
Taking $k=1$ and swapping $m$ and $n$ gives the
Cayley--Sylvester formula, while taking $k=1$ and swapping $k$ and~$n$ gives
$(s_{(n)} \circ s_{(m)})(x_1,x_2) = (s_{(1^n)} \circ s_{(m+n-1)})(x_1,x_2)$. 
In \cite{LangleyRemmel} Langley and Remmel used
symmetric functions methods to
determine the multiplicities in
$s_\nu \circ s_\mu$ of the Schur functions 
$s_{(mn-r,1^r)}$, $s_{(mn-r-s,s,1^{r})}$ and $s_{(mn-r-2t,2^t,1^r)}$,
for any partition $\mu$ of $m$.
Giannelli \cite[Theorem~1.2]{Giannelli} later used character-theoretic methods to
determine the multiplicities of a much
larger class of `near hook' constituents of $s_{(n)} \circ s_{(m)}$.

For sufficiently small partitions $\nu$ and $\mu$, the plethysm
$s_\nu \circ s_{\mu}$ 
can readily be calculated using any of the computer algebra systems {\sc Magma} \cite{Magma}, 
{\sc Gap}~\cite{GAP4}
or {\sc Symmetrica} \cite{Symmetrica}. 
A new algorithm for computing $s_{(n)} \circ s_{(m)}$ was given
in \cite[Proposition 5.1]{EvseevPagetWildon}, and used to verify Foulkes' Conjecture
(see \cite[end~\S 1]{Foulkes})
for all $m$ and $n$ such that $m + n \le 19$.

 Applying the $\omega$ involution (see \cite[Ch.~I, Equation~(2.7)]{Macdonald}) gives  further results for the plethysms $s_\nu \circ s_{(1^m)}$, via the relation 
 \begin{equation}\label{eqn:omega}
\omega(s_\nu \circ s_{(m)}) = \begin{cases} s_{\nu'} \circ s_{(1^m)} & \text{ if $m$ is odd} \\ s_{\nu} 
\circ s_{(1^m)} & \text{ if $m$ is even,} \end{cases} 
 \end{equation} 
 which follows from \cite[Ch.~I, Equation~(3.8) and \S 8, Example 1(a)]{Macdonald}. This equation
 is reformulated in terms of modules and characters in Section~\ref{subsec:Q}.

Finally we note that the lexicographically greatest 
 constituent of $s_\nu \circ s_\mu$ was
determined by Iijima in \cite[Theorem 4.2]{Iijima}, confirming a conjecture of Agaoka 
\cite[Conjecture 1.2]{Agaoka}. 
We give a short proof of the special cases of
Iijima's result when $\mu = (m)$ or $\mu=(1^m)$ in Section~\ref{sec:corollaries} below.

\section{Specht modules and plethysm}\label{sec:plethysm}

In this section we recall a standard construction of Specht modules as modules
defined by generators and relations. We then give a functorial interpretation of plethysm 
for
 categories of modules for symmetric groups. This leads to an explicit construction of
modules affording the characters $\phi^{(m^n)}_\nu$ and $\sgn_{S_{mn}} \!\times\, \phi^{(m^n)}_\nu$. 

\subsection{Garnir elements} \label{subsec:Garnir}
Let $\lambda$ be a partition of $r \in \N$.
We use the standard definition  \cite[Definition 4.3]{James} of the rational Specht module $S^\lambda$ 
as the $\Q S_r$-submodule of the Young permutation module $M^\lambda$ spanned by the
$\lambda$-polytabloids $e_t$ for~$t$ a $\lambda$-tableau. It is well known that 
$S^\lambda$ affords the irreducible character~$\chi^\lambda$.

Following Fulton (see \cite[Chapter 7, Section 4]{FultonYT}), we define a $\lambda$-\emph{column tabloid}
to be an equivalence class of $\lambda$-tableaux up to column equivalence. We denote 
the column
tabloid corresponding to a tableau $t$ by $|t|$ and represent it by omitting the
horizontal lines from the representative $t$.
The symmetric group acts in an obvious
way on the set of $\lambda$-column tabloids: let $U \cong M^{\lambda'}$ denote the corresponding
permutation module for $\Q S_r$. We define
$\widetilde{M}^\lambda = \sgn_{S_r} \cotimes U$. (This is equivalent to Fulton's definition using
orientations.) By a small abuse of notation we shall write $|t|$ for the basis element of $\widetilde{M}^\lambda$ 
corresponding to the $\lambda$-column tabloid $t$.
There is a canonical surjective
homomorphism
 of $\Q S_{mn}$-modules $\widetilde{M}^\lambda \rightarrow S^\lambda$ defined
by $|t| \mapsto e_t$.

It follows from the corollary on page 101 of \cite{FultonYT} and
the proof of Theorem~8.4 of \cite{James} 
that the kernel of the canonical
surjection $\widetilde{M}^\lambda \rightarrow S^\lambda$ is spanned by all elements of $\widetilde{M}^\lambda$ of the form
\begin{equation}\label{eq:kernelElts} |t| \sum_{\sigma \in S_{X \cup Y}}
\sigma \sgn(\sigma) 
\end{equation}
where $t$ is a $\lambda$-tableau,
$X$ is a subset of set of all entries in column $i$ of~$t$ and~$Y$ is a subset of the entries in column $i+1$ of $t$
such that $|X| + |Y| > \lambda_i'$. By
Exercise~16 on page 102 of \cite{FultonYT},
we need only consider the case when $Y$ is a singleton set; note that
this result requires that the ground field has characteristic zero.
An easy calculation now shows that, if $t$ is any fixed $\lambda$-tableau, 
then the kernel is generated, as a $\mathbb{Q}S_{mn}$-module, 
by the $t$-\emph{Garnir elements}
$|t| \sum_{\sigma \in S_{X \cup \{y\}}} \sigma \sgn(\sigma)$, where $X$ is the set of entries
in column $i$ of $t$ and $y$ is the entry at the top of column $i+1$ of $t$.
(This term is not standard, but will be convenient in this paper.)

\subsection{The plethysm functor $P$} \label{subsec:P}

Let $m$, $n \in \N$ and let~$\nu$ be a partition of~$n$. Let $S_{mn}$ act naturally on the set $\Omega$ of size $mn$.
Given a module $V$ for $\Q S_n$ we define
\[ P(V) = \bigl( \Inf_{S_n}^{S_m \wr S_n} V \bigr) \Ind_{S_m \wr S_n}^{S_{mn}}. \]
Since $P$ is the composition of an 
inflation and an induction functor, $P$ is an exact functor
from the category of $\Q S_n$-modules to the category of $\Q S_{mn}$-modules. 
By definition $P(S^\nu)$ affords the irreducible character $\phi^{(m^n)}_\nu$.

We now give an explicit model for the modules $P(M^\nu)$, $P(\widetilde{M}^\nu)$
and $P(S^\nu)$. These modules  have bases defined using tableaux, tabloids and column tabloids
with entries taken from the set of $m$-subsets of the set $\Omega$ of size~$mn$; 
we shall refer to these objects as \emph{set-tableaux}, \emph{set-tabloids} and \emph{column set-tabloids}. 
Let $S_m \wr S_n \le S_{mn}$ have $\{\Delta_1, \ldots, \Delta_n\}$ as a system of blocks of imprimitivity.
 As a concrete module isomorphic to 
 $\Inf_{S_n}^{S_m \wr S_n} M^\nu$, we take the rational vector space $W$ with basis the set of 
 set-tabloids of shape $\nu$ with entries from  
 $\{\Delta_1, \ldots, \Delta_n\}$. 
Let $W'$ denote the rational vector space with basis the set of all set-tabloids of shape $\nu$
such that the union of all the $m$-subsets
appearing in each set-tabloid is $\Omega$. Then $W'$
is a $\Q S_{mn}$-module of dimension $|S_{mn}|/|S_m \wr S_n| \dim W$, generated by
the $\Q(S_m \wr S_n)$-submodule~$W$. Hence $W' \cong W\ind_{S_m \wr S_n}^{S_{mn}}$
and so $W' \cong P(M^\nu)$. 
By the functoriality of $P$ the canonical inclusion map
$S^\nu \hookrightarrow M^\nu$ induces a canonical inclusion 
\[ P(S^\nu) \hookrightarrow P(M^\nu).\]
An entirely analogous construction with set-tableaux and column set-tabloids gives modules
isomorphic to $P(\Q S_n)$ and $P(\widetilde{M}^\nu)$, respectively, with canonical quotient maps
\[ 
P(\widetilde{M}^\nu) \twoheadrightarrow P(S^\nu). \] 
We illustrate this construction in Section~\ref{subsec:ex} below.

\subsection{The signed plethysm functor $Q$} \label{subsec:Q}

Let $\widetilde{\sgn}$ denote the unique $1$-dimensional module for $S_m \wr S_n$ that restricts 
to the module $\sgn \otimes \cdots \otimes \sgn$ of the base group 
$S_m \times \cdots \times S_m$ and on which the complement $S_n$ acts trivially. 
Given a module $V$ for $\Q S_n$ we define
\[ Q(V)= 
(\widetilde{\sgn} \otimes \Inf_{S_n}^{S_m \wr S_n} V)\Ind_{S_m \wr S_n}^{S_{mn}}. \] 
Again $Q$ is an exact functor from the category of $\Q S_n$-modules to the category of $\Q S_{mn}$-modules.

We define $\psi^{(m^n)}_\nu$ to be the character of $Q(S^\nu)$. 
%
The twisted Foulkes characters $\phi^{(m^n)}_\nu$ are related to the characters $\psi_\nu^{(m^n)}$ via a sign-twist. 
We have
\[
 \sgn_{S_{mn}} \otimes P(V)   =  
\left(  \sgn \Res^{S_{mn}}_{S_m \wr S_n} \otimes \Inf_{S_n}  V \right)\Ind_{S_m \wr S_n}^{S_{mn}}.
\]
The restriction of $\sgn_{S_{mn}}$ to $S_m \wr S_n$ is $\widetilde{\sgn}$ if $m$ is even and $\widetilde{\sgn}\, {} \otimes  {}
\Inf_{S_n}^{S_m \wr S_n} \sgn_{S_n}$ if $m$ is odd.
Therefore 
\begin{equation}\label{eqn:QP}
\sgn_{S_{mn}}\cotimes P(V)   \cong \begin{cases} Q(V) & \text{ if $m$ is even} \\
Q(\sgn_{S_n}\cotimes V) & \text{ if $m$ is odd.} \end{cases}\end{equation}
Using the isomorphism  
\begin{equation}\label{eqn:Specht_sign_twist}
\sgn_{S_n}  \cotimes S^{\nu}  \cong (S^{\nu'})^*
\end{equation}
(see, for example,
\cite[Theorem~6.7]{James}), and that Specht modules are self-dual over the rationals (see \cite[Theorem~4.12]{James}),
we obtain the reformulation of Equation~(\ref{eqn:omega}) for characters:
\begin{equation}\label{eqn:sign_twist}
  \sgn_{S_{mn}}\times \, \phi^{(m^n)}_\nu  = \begin{cases} \psi^{(m^n)}_\nu & \textrm{if $m$ is even}\\ 
 \psi^{(m^n)}_{\nu'} & \textrm{if $m$ is odd.} \end{cases} 
\end{equation}


\subsection{Connection with Schur functors}\label{subsec:Schur}
We remark very briefly on an alternative definition of these functors. 
Let $\Delta^\lambda$ be the Schur functor corresponding to the partition $\lambda$ (see \cite[page 76]{FultonHarris} or
 \cite[page 273]{Procesi}).
Let $E$ be a rational vector space
of dimension at least $mn$. If $F$ is the functor defined in \cite[Section 6.1]{Green}
from polynomial representations of $\GL(E)$
of degree $r$ to representations of $S_r$ then, by \cite[I, Appendix A, (6.2)]{Macdonald}, 
$F\bigl( \Delta^\nu (\Sym^m\! E) \bigr) = P(S^\nu)$, corresponding to the plethysm $s_\nu \circ s_{(m)}$,
and $F\bigl( \Delta^\nu (\bigwedge^m E) \bigr) = Q(S^\nu)$, corresponding to the plethysm $s_\nu \circ s_{(1^m)}$.
We use this interpretation of $P$ and $Q$ in Section~\ref{subsec:rectangular} below.

\section{Further preliminary results and an example}\label{sec:closed}

\subsection{Closed set families}\label{subsec:closed}
Let $A$ and $B$ be $m$-subsets of $\N$. Let $a_r$ be the $r$th smallest
element of $A$, and  let $b_r$ be the $r$th smallest element of $B$. 
We say that $B$ \emph{majorizes} $A$, and
write $A \preceq B$, if $a_r \le b_r$ for all $r$.
We say that a set family $\mathcal{P}$ of shape $(m^n)$ 
is \emph{closed} if whenever $B \in \mathcal{P}$ and $A$ is an $m$-subset of $\N$
such that $A \preceq B$, then $A \in \mathcal{P}$. 
We say that a set family tuple $(\mathcal{P}_1, \ldots, \mathcal{P}_k)$
is \emph{closed} if $\mathcal{P}_j$ is closed for each $j$.
It is easily seen that closed set families and closed set family tuples have well-defined types.

If  $(\mathcal{P}_1, \ldots, \mathcal{P}_k)$ is a minimal set family tuple then it is closed.  For if not
 there is a set family $\mathcal{P}_j$, a set $A \in \mathcal{P}_j$
and  an element $i+1 \in A$, such that  the set $B = A \backslash \{i+1\}\,\cup \{i\}$ is 
not in $\mathcal{P}_j$. A new set family tuple can be formed by replacing $A$ by $B$ in $\mathcal{P}_j$
and this process repeated until a closed set family tuple $(\mathcal{P}'_1, \ldots, \mathcal{P}'_k)$ is obtained: this 
 set family tuple
has a well-defined type. By construction, 
  $\mathcal{P}'_j$ has the same shape as $\mathcal{P}_j$ for each $j$, and the type of $(\mathcal{P}'_1, \ldots, \mathcal{P}'_k)$ is strictly dominated by the type of $(\mathcal{P}_1, \ldots, \mathcal{P}_k)$, 
  contradicting minimality.
 This argument also shows that if $(\mathcal{P}_1, \ldots, \mathcal{P}_k)$ is a minimal set family tuple then  each set family $\mathcal{P}_j$ is minimal.

 Closed  multiset family tuples are defined analogously and the same argument shows that 
 minimal  multiset family tuples   are closed.

\subsection{Symbols}\label{subsec:symbols}
When defining maps from $S^\lambda$ or from $\widetilde{M}^\lambda$,
it will be convenient to think of $S_{mn}$ as the symmetric group on the set
$\Omega^\lambda$ whose elements are the formal \emph{symbols}
$i_j$ for $i$ and $j$ such that $1 \le i \le \lambda_1$ and $1 \le j \le \lambda_i'$. 
We say that $i_j$ has \emph{number} $i$ and \emph{index} $j$. 
Let $t_\lambda$ be the $\lambda$-tableau such that column $i$ of $t_\lambda$ has entries $i_1, \ldots, i_{\lambda_i'}$
when read from top to bottom. 
Let $C(t_\lambda)$ be the column stabilising subgroup of $t_\lambda$; note
that $C(t_\lambda)$ permutes the indices of the symbols in $\Omega^\lambda$, 
while leaving the numbers fixed. 
Let $b_{t_\lambda} = \sum_{\sigma \in C(t_\lambda)} \sigma \sgn(\sigma)$.

\subsection{Example}\label{subsec:ex} This example illustrates the definitions so far,
and  
many of the ideas in the proofs of Theorem~\ref{thm:min} and Theorem~\ref{thm:max} to follow. Let $m=2$,
let $\nu = (2,1,1)$ and let
$\mathcal{P}=( \{ \{1,2\}, \{1,3\}, \{1,4\} \}, \{ \{1,2\} \} )$ be the minimal
set family tuple of type $\lambda=(4,2,1,1)$ seen in the introduction.
 We identify $S_8$ with the symmetric
group on the set $\Omega^{(4,2,1,1)} = \{ \oa, \ob, \oc, \od, \pa, \pb, \qa, \ra \}$
and choose $S_2 \wr S_4 \le S_8$ to 
have blocks of imprimitivity $\{\oa, \pa\}$, $\{\ob, \qa\}$,
$\{\oc, \ra\}$, $\{\od, \pb\}$. 
Let $T$ be the set-tableau
\[
\arraycolsep=2pt\arrayrulewidth=0.5pt
\newcommand{\rl}{\rule[-6pt]{0pt}{18pt}}
\quad
\begin{array}{|c|c|} 
\hline \{\oa,\pa\} & \{\od,\pb\}\rl \\ 
\hline \{\ob,\qa\}\rl \\ 
\cline{1-1} \{\oc,\ra\}\rl \\ 
\cline{1-1} 
\end{array} .\]
The column set-tabloid $|T|$ generates $\Inf_{S_4}^{S_2 \wr S_4} \widetilde{M}^\nu$ as a 
$\Q(S_2 \wr S_4)$-module and $P(\widetilde{M}^\nu)$ as a $\Q S_8$-module.  
For example
\[
|T|\hskip1pt (\oa,\ob) =
\arraycolsep=2pt\arrayrulewidth=0.8pt 
\begin{array}{|c|c|} 
\{\ob,\pa\} & \{\od,\pb\} \\ 
\{\oa,\qa\} \\ 
\{\oc,\ra\} 
\end{array} \; = - \;
\begin{array}{|c|c|} 
\{\oa,\qa\} & \{\od,\pb\} \\ 
\{\ob,\pa\} \\ 
\{\oc,\ra\} 
\end{array}.
\]
There is a unique 
homomorphism of $\Q S_8$-modules $\widetilde{M}^{(4,2,1,1)} \rightarrow 
P(\widetilde{M}^{(2,1,1)})$ sending $|t_{(4,2,1,1)}|$ to $|T|b_{t_{(4,2,1,1)}}$. 
We shall see in the proof of Proposition~\ref{prop:evenHoms} below that the kernel of 
this homomorphism contains all the $t_{(4,2,1,1)}$-Garnir elements. Hence
there is a well-defined homomorphism of $\Q S_8$-modules 
$S^{(4,2,1,1)} \rightarrow P(\widetilde{M}^{(2,1,1)})$
defined by $e_{t_{(4,2,1,1)}} \mapsto |T|b_{t_{(4,2,1,1)}}$.
After composition with the canonical surjection
$P(\widetilde{M}^{(2,1,1)}) \rightarrow P(S^{(2,1,1)})$ the image of $e_{t_{(4,2,1,1)}}$ is 
$e_T b_{t_\lambda} \in P(S^{(2,1,1)}) \subseteq P(M^{(2,1,1)})$. 
As we argue in Lemma~\ref{lemma:evenHoms} below,
the coefficient of the tabloid $\{ T \}$ in
$e_T b_{t_\lambda}$ is $1$, and so this map is non-zero.
Hence $\langle \phi^{(2^4)}_{(2,1,1)}, \chi^{(4,2,1,1)} \rangle \ge 1$. 

Observe that if $T$ is a set-tableau having 
an entry containing symbols~$i_j$ and $i_k$ with $j\not=k$ then $|T|(i_j,i_k) = 1$,
whereas $e_{t_\lambda}(i_j,i_k) = -1$. Thus the entries of $T$
must come from set families. (This remark is made precise in the proof of Proposition~\ref{prop:characterToSetFamily} below.)
We shall see in Section~\ref{sec:psi} that the maximal constituents of $\phi_{(2,1,1)}^{(2^4)}$ 
are determined by homomorphisms into $Q(S^{(2,1,1)}) \cong \sgn_{S_8} \cotimes P(S^{(2,1,1)})$.
In this setting, thanks to the sign-twist, the two signs agree. 
This gives one indication why set-tableaux
with entries given by multiset families, rather than set families, are relevant to maximal constituents.

\section{Proof of Theorem~\ref{thm:min} for $m$ even}\label{sec:evenProof}

Fix an even number $m$. Let $\nu$ be a partition of $n$ with first part $k$. 
The proof of Theorem~\ref{thm:min} for even $m$ has two steps. In the first we construct an explicit
homomorphism $S^\lambda \rightarrow P(S^\nu)$ for each closed set family tuple of type~$\lambda$.
We then use these homomorphisms to show that the minimal constituents of the character
$\phi^{(m^n)}_\nu$ are as claimed in the theorem. 
We must begin with one more definition.

Let $(\mathcal{P}_1, \ldots, \mathcal{P}_k)$ be a  set family tuple of type $\lambda$
such that $\mathcal{P}_j$ has shape $(m^{\nu_j'})$ for each $j$. 
Let $\mathcal{A}{(\mathcal{P}_1, \ldots, \mathcal{P}_k)}$ be the set of all ordered pairs 
$(j, B)$ such that $1 \le j \le k$ and $B \in \mathcal{P}_j$. We totally order 
$\mathcal{A}(\mathcal{P}_1, \ldots, \mathcal{P}_k)$ so that
$(i,A) \le (j,B)$ if and only if $i < j$ or $i = j$ and $A \le B$, where the final
inequality refers to the lexicographic order on sets.

\begin{definition}{ \ }\label{def:coltab}
\begin{defnlist}
\item
The \emph{column set-tableau} corresponding to $(\mathcal{P}_1, \ldots, \mathcal{P}_k)$
is the unique set-tableau $T$ of shape $\nu$ such that if $\mathcal{P}_j = \{A_1, \ldots, A_{\nu_j'} \}$ 
then the entries in column~$j$ of $T$ are obtained by
appending indices to the numbers in the sets $A_1, \ldots, A_{\nu_j'}$,
listing the sets in lexicographic order and
choosing indices in the order specified by the order
 on $\mathcal{A}{(\mathcal{P}_1, \ldots, \mathcal{P}_k)}$. 
 
\item The \emph{column set-tabloid} 
corresponding to $(\mathcal{P}_1, \ldots, \mathcal{P}_k)$ is 
 $|T| \in P(\widetilde{M}^{\nu})$,
 where $T$ is the column set-tableau corresponding to $(\mathcal{P}_1, \ldots, \mathcal{P}_k)$ .
 \end{defnlist}
\end{definition}
Observe that the union of the entries in the column set-tableau 
corresponding to $(\mathcal{P}_1, \ldots, \mathcal{P}_k)$ is the set $\Omega^\lambda$.
For example, the set-tableau $T$ in Section~\ref{subsec:ex} is the column set-tableau
corresponding to the set family tuple $\bigl( \bigl\{ \{1,2\}, \{1,3\}, \{1,4\} \bigr\},
\big\{ \{ 1,2\} \bigr\} \bigr)$.

Let $(\mathcal{P}_1, \ldots, \mathcal{P}_k)$ be a closed set family tuple of type $\lambda$
such that $\mathcal{P}_j$ has shape $(m^{\nu_j'})$ for each $j$. 
Let $|T| \in P(\widetilde{M}^\nu)$ be the column set-tabloid corresponding
to $(\mathcal{P}_1, \ldots, \mathcal{P}_k)$.
Let $t_\lambda$ be the $\lambda$-tableau defined in Section~\ref{subsec:symbols}, and let
\[f_{(\mathcal{P}_1, \ldots, \mathcal{P}_k)} : \widetilde{M}^\lambda \rightarrow
P(\widetilde{M}^\nu)\]
 be the unique 
$\Q S_{mn}$-homomorphism such that
\[ |t_\lambda|f_{(\mathcal{P}_1, \ldots, \mathcal{P}_k)} = |T| b_{t_\lambda}. \]

\begin{proposition}\label{prop:evenHoms}
The kernel of $f_{(\mathcal{P}_1, \ldots, \mathcal{P}_k)}$ contains every $t_\lambda$-Garnir
element.
\end{proposition}

\begin{proof}
Let $1 \le i < \lambda_1$ and let $X = \{i_1, \ldots, i_{\lambda_i'} \}$ 
be the set of entries in column~$i$ of $t_\lambda$.
We have
 \[ |t_\lambda| G_{X\scup \{(i+1)_1\}} f_{(\mathcal{P}_1, \ldots, \mathcal{P}_k)}
 = 
 |T| \sum_{\tau \in C(t_\lambda)} \tau  G_{X \scup \{(i+1)_1\}} \sgn(\tau). \]  
To prove that the right-hand side is zero we shall
construct an involution on $C(t_\lambda)$, denoted 
$\tau \mapsto \tau^\star$, with the following two properties:
\begin{itemize}
\item[(a)] if $\tau = \tau^\star$ then $|T| \tau G_{X \scup \{ (i+1)_1 \}} = 0$, 
\item[(b)] if $\tau \not= \tau^\star$ then $|T| \bigl(\tau \sgn(\tau) + 
\tau^\star \sgn(\tau^\star)\bigr) G_{X \scup \{ (i+1)_1 \} } = 0$.
\end{itemize}

 Let $\tau \in C(t_\lambda)$. Consider $|T| \tau$. If there exists a symbol  $i_x \in X$   such that there is an entry in $|T| \tau$ containing both $i_x$ and $(i+1)_1$, then we have
 $|T| \tau (1 - (i_x, (i+1)_1)) = 0$. Taking coset representatives for $\langle (i_x, (i+1)_1) \rangle \le S_{X \scup \{(i+1)_1\}}$, it follows that $|T| \tau G_{X \scup \{(i+1)_1\}} = 0$. Hence if we define $\tau^\star = \tau$ in this case then~(a) holds.
 
 Now suppose that no entry in $|T| \tau$ contains both $(i+1)_1$ and a symbol $i_x \in X$.
Let the entry of $|T|\tau$ containing $(i+1)_1$ be 
\begin{align*} 
 B_{(i+1)_1} &= \{c(1)_{b(1)}, \ldots, c(m-1)_{b(m-1)},  (i+1)_1 \}. 
\intertext{Suppose that $B_{(i+1)_1}$ lies in column $j$ of $|T|\tau$.
This column is defined using the set family $\mathcal{P}_j$. Since $\mathcal{P}_j$ is closed, there exists unique symbols $c(1)_{a(1)}, \ldots, c(m-1)_{a(m-1)}$ and $i_u$ such that the multiset}
A_{(i+1)_1} &= \{c(1)_{a(1)}, \ldots, c(m-1)_{a(m-1)},  i_u \} 
\end{align*} 
is also an entry in column $j$ of $|T|\tau$.
 Define 
 \[ \pi =  (c(1)_{a(1)}, c(1)_{b(1)}) \ldots (c(m-1)_{a(m-1)}, c(m-1)_{b(m-1)}) \in C(t_\lambda) \]  
 and define $\tau^\star = \tau \pi$. 
 Since the column set-tabloids $|T|\tau$ and $|T|\tau^\star$ differ only in indices attached to numbers other than $i$ and $i+1$, we have $\tau^{\star\star} = \tau$.
  Since~$m$ is even we have
  $\sgn(\tau) = -\sgn(\tau^\star)$ and since $\pi (i_u, (i+1)_1)$ swaps two entries in column $j$ of $|T|\tau$ we have \[ |T| \tau^\star (i_u, (i+1)_1) = |T| \tau  \pi (i_u, (i+1)_1) = - |T| \tau. \]
 Using this relation to eliminate $\tau^\star$ we obtain
 \[ \bigl(  |T| \tau \sgn(\tau) + |T| \tau^\star \sgn(\tau^\star) \bigr) \bigl( 1 - (i_u,(i+1)_1) \bigr)=0.\]
Hence $\bigl( |T| \tau \sgn(\tau) + |T| \tau^\star \sgn(\tau^\star) \bigr)
G_{X \scup \{(i+1)_1\}} = 0$, as required in (b).
\end{proof}

It now follows from Section~\ref{subsec:Garnir} that 
$f_{(\mathcal{P}_1, \ldots, \mathcal{P}_k)}$ induces a homomorphism
$S^\lambda \rightarrow P(\widetilde{M}^\nu)$. Let
\[\bar{f}_{(\mathcal{P}_1, \ldots, \mathcal{P}_k)} : S^\lambda \rightarrow P(S^\nu)\] 
denote the composition
of this homomorphism with the canonical quotient map $P(\widetilde{M}^\nu)
\rightarrow P(S^\nu)$. Thus $\bar{f}_{(\mathcal{P}_1, \ldots, \mathcal{P}_k)}$ is defined on
the generator $e_{t_\lambda}$ of $S^\lambda$ by
\[e_{t_\lambda} \bar{f}_{(\mathcal{P}_1, \ldots, \mathcal{P}_k)} =
e_T b_{t_\lambda}.\]

\begin{lemma}\label{lemma:evenHoms}
The homomorphism $\bar{f}_{(\mathcal{P}_1, \ldots, \mathcal{P}_k)} : S^\lambda \rightarrow P(S^\nu)$ 
is non-zero.
\end{lemma}

\begin{proof}
Since $b_{t_\lambda}$ 
permutes the indices attached to numbers,
while leaving the numbers fixed, it is clear
that the coefficient of the set-tabloid $\{T\}$ in $\{T\}b_{t_\lambda}$ is~$1$.
This is also the coefficient of $\{T\}$ in $e_T b_{t_\lambda}$.
\end{proof}

We summarize the results proved so far in the following corollary.
We show in Section~\ref{subsec:further} that this corollary gives
constituents of $\phi^{(m^n)}_\nu$ beyond those predicted by Theorem~\ref{thm:min}.

\begin{corollary}\label{cor:even}
Let $m$ be even and let $\nu$ be a partition of $n$ with first part~$k$. 
If there is a closed  set family tuple $(\mathcal{P}_1,\ldots,\mathcal{P}_k)$
of type $\lambda$ such that $\mathcal{P}_i$ has shape $(m^{\nu_i'})$ for each $i$,
then $\langle \phi^{(m^n)}_\nu, \chi^\lambda \rangle \ge 1$. 
\end{corollary}

\begin{proof}
This follows immediately from Proposition~\ref{prop:evenHoms} and
Lemma~\ref{lemma:evenHoms}
\end{proof}

The final ingredient in the proof of Theorem~\ref{thm:min} in the case when $m$ is even is a result that goes in the opposite direction
to Corollary~\ref{cor:even}.

\begin{proposition}\label{prop:characterToSetFamily}
Let $m$ be even and let $\nu$ be a partition of $n$ with first part~$k$.
If $\chi^\mu$ is a  constituent of $\phi^{(m^n)}_\nu$ 
then there is a  set family tuple $(\mathcal{R}_1,\ldots,
\mathcal{R}_k)$ of type $\mu$ such
that $\mathcal{R}_j$ has shape $(m^{\nu_j'})$ for each $j$.
\end{proposition}

\begin{proof}
Let $\zeta$ be the character of $P(\widetilde{M}^\nu)$. We have
$\langle \zeta, \chi^\mu \rangle \ge \langle \phi^{(m^n)}_\nu, \chi^\mu \rangle \ge 1$.
Hence there is a non-zero $\Q S_{mn}$-module homomorphism
$f : S^\mu \rightarrow P(\widetilde{M}^\nu)$. Identify $S_{mn}$ with the symmetric group on the symbols $\Omega^\mu$.  Let $T$ be a set-tableau
such that the coefficient of $|T|$ in $e_{t_\mu} f$ is non-zero.
Let $i_j$ and~$i_{j'}$ be symbols appearing in $t_\mu$. If
there is an entry in $T$ containing both $i_j$ and $i_{j'}$ then we have
$|T| (i_j, i_{j'}) = |T|$, whereas $e_{t_\mu} (i_j, i_{j'}) = -e_{t_\mu}$,
a contradiction.
Now suppose that there is a column of $T$ containing entries 
$\{ c(1)_{a(1)}, \ldots, c(m)_{a(m)} \}$ and $\{ c(1)_{b(1)}, \ldots, c(m)_{b(m)} \}$
that are equal up to the indices attached to numbers. Let
\[ \rho = (c(1)_{a(1)}, c(1)_{b(1)}) \ldots (c(m)_{a(m)}, c(m)_{b(m)}). \] 
Since $\rho$ swaps two entries in the same column
of $|T|$, we have $|T|\rho = - |T|$. But since $\rho$ is even, $e_{t_\mu} \rho = 
e_{t_\mu}$, so again we have a contradiction. It follows that removing
the indices attached to the numbers in
column $j$ of $|T|$ gives a set family $\mathcal{R}_j$ of shape  $(m^{\nu_j'})$. 
Since the union of the entries in $|T|$ is $\Omega^\mu$,
the set family tuple $(\mathcal{R}_1,\ldots,\mathcal{R}_k)$ has type $\mu$, as required.
\end{proof}

We are now ready to prove Theorem~\ref{thm:min} for even values of $m$.
Suppose that  $(\mathcal{P}_1,\ldots,
\mathcal{P}_k)$ is a minimal set family tuple 
 of type $\lambda$ such that each $\mathcal{P}_j$ 
has shape $(m^{\nu_j'})$. We saw in Section~\ref{subsec:closed} that any minimal
 set family tuple   is closed. Hence, by Corollary~\ref{cor:even},
$\langle \phi^{(m^n)}_\nu, \chi^\lambda \rangle \ge 1$. 
If $\mu$ is a partition of~$mn$ such that $\lambda \unrhd \mu$
and $\langle \phi^{(m^n)}_\nu, \chi^\mu \rangle \ge 1$ then,
by Proposition~\ref{prop:characterToSetFamily}, there
is a set family tuple   $(\mathcal{R}_1,\ldots,
\mathcal{R}_k)$  of type $\mu$
such that $\mathcal{R}_j$ has shape $(m^{\nu_j'})$ for each $j$.
But $(\mathcal{P}_1,\ldots,
\mathcal{P}_k)$ is minimal, so we must have $\lambda = \mu$.
Hence $\chi^\lambda$ is a minimal constituent of $\phi^{(m^n)}_\nu$.

Conversely suppose that $\chi^\lambda$ is a minimal constituent
of $\phi^{(m^n)}_\nu$. By Proposition~\ref{prop:characterToSetFamily}
there is a  set family tuple 
 $(\mathcal{R}_1,\ldots,
\mathcal{R}_k)$  of type $\lambda$
such that $\mathcal{R}_j$ has shape $(m^{\nu_j'})$ for each $j$.
Hence there is a minimal set family tuple 
$(\mathcal{P}_1,\ldots,
\mathcal{P}_k)$ of type $\mu$ where $\lambda \unrhd \mu$
such that $\mathcal{P}_i$  has shape $(m^{\nu_j'})$ for 
each $j$. Once again we have
$\langle \phi^{(m^n)}_\nu, \chi^\mu \rangle \ge 1$.
But $\chi^\lambda$ is a minimal constituent of $\phi^{(m^n)}_\nu$
so we must have $\lambda = \mu$. Hence  $(\mathcal{R}_1,\ldots,
\mathcal{R}_k)$ is a minimal  set family tuple. This completes
the proof.

\section{Proof of Theorem~\ref{thm:min} for $m$ odd}\label{sec:oddProof}

Theorem~\ref{thm:min} can be proved for odd values of $m$ by modifying the proof in the case of $m$ even,
following the same
logical structure of Section~\ref{sec:evenProof}. We give the required changes in detailed outline.

 Let $\nu$ be a partition of $n$ 
with precisely $k$ parts
 and let $(\mathcal{P}_1, \ldots,
\mathcal{P}_k)$ be a  set family tuple of type $\lambda$ such
that $\mathcal{P}_j$ has shape $(m^{\nu_j})$ for each $j$. 
Define the totally ordered 
set $\mathcal{A}(\mathcal{P}_1,\ldots,\mathcal{P}_k)$ of pairs $(j,X)$ with
$1 \le j \le k$ and $X \in \mathcal{P}_j$ as before. We define the \emph{row set-tableau} $T$
and the \emph{set-tabloid} $\{T\}$ corresponding to  $(\mathcal{P}_1, \ldots, \mathcal{P}_k)$
by analogy with Definition~\ref{def:coltab}. Thus $T$ and $\{T\}$ have shape $\nu$, the entries in row $j$ of $T$ and $\{T\}$
are determined by the order on $\mathcal{A}(\mathcal{P}_1,\ldots,\mathcal{P}_k)$, 
and the union of all the entries in $T$ or $\{T\}$ is $\Omega^\lambda$.

Let $(\mathcal{P}_1, \ldots, \mathcal{P}_k)$ be a closed  set family tuple  of 
type $\lambda$ as above and let $\{T\} \in
P(M^\nu)$ be the corresponding set-tabloid. 
We define
\[ g_{(\mathcal{P}_1, \ldots, \mathcal{P}_k)} : \widetilde{M}^\lambda \rightarrow
P(M^\nu) \]
by $|t_\lambda| g_{(\mathcal{P}_1, \ldots, \mathcal{P}_k)}
= \{T \} b_{t_\lambda}$. We now follow Section~\ref{sec:evenProof}, making the following changes.

\begin{itemize}
\item[(1)] \emph{Proposition 5.2.}
The proof of the analogue of Proposition~\ref{prop:evenHoms}
 goes through almost unchanged.
Now swapping two entries in the same row
of a  set tabloid $\{T\}$ leaves the sign unchanged, but
the permutation $\pi$ is even. The pattern of cancellation in
$\bigl( \{T\}\tau \sgn(\tau) + \{T\}\tau^\star \sgn(\tau^\star) \bigr)
G_{X \scup \{(i+1)_1\}}$ is therefore the same.

\smallskip
\item[(2)] \emph{Definition of homomorphisms into $P(S^\nu)$.}
Let $\{u \} \in M^\nu$ be a fixed tabloid. By \cite[Equation (6.8)]{James}, there is a surjective $\Q S_n$-homomorphism
$M^\nu \rightarrow \sgn_{S_n} \otimes S^{\nu'}$ defined 
on the generator $\{u\}$ by $\{u\} \mapsto w \otimes e_{u'}$, where $u'$ is the tableau conjugate to $u$
and $w$ generates $\sgn_{S_n}$.
Applying $P$ gives a canonical quotient map $P(M^\nu) \rightarrow P(\sgn_{S_n} \otimes S^{\nu'})$.
Composing the map induced by $g_{(\mathcal{P}_1, \ldots, \mathcal{P}_k)}$ on
$S^\lambda$ 
with this surjection gives a homomorphism
$\bar{g}_{(\mathcal{P}_1, \ldots, \mathcal{P}_k)} : S^\lambda
\rightarrow P(\sgn_{S_n}\otimes  S^{\nu'} )$ 
sending $e_{t_\lambda}$ to $(w \otimes e_{T'})b_{t_\lambda}$,
where $T'$ is the conjugate set-tableau to $T$. 
The isomorphisms $\sgn_{S_n} \otimes S^{\nu'} \cong (S^\nu)^* \cong S^\nu$ 
seen in Equation~\eqref{eqn:Specht_sign_twist} and the following remark in Section~\ref{subsec:Q} 
now identify the codomain
of $\bar{g}_{(\mathcal{P}_1, \ldots, \mathcal{P}_k)}$ with $P(S^\nu)$. 

\smallskip

\item[(3)] \emph{Lemma~\ref{lemma:evenHoms}.}
Thinking of the codomain of $\bar{g}_{(\mathcal{P}_1, \ldots, \mathcal{P}_k)}$
as a submodule of $P(\sgn_{S_n} \otimes M^{\nu'})$ it follows by looking at 
the coefficient
of $\{T'\}$ in $e_{t_\lambda} \bar{g}_{(\mathcal{P}_1, \ldots, \mathcal{P}_k)}$  that 
$\bar{g}_{(\mathcal{P}_1, \ldots, \mathcal{P}_k)}$ is non-zero. 

\smallskip

\item[(4)] \emph{Corollary~\ref{cor:even}.} The analogous result holds with the same proof.

\smallskip

\item[(5)] \emph{Proposition~\ref{prop:characterToSetFamily}.} 
The use of characters at the start of the proof can be avoided as follows:
given a non-zero homomorphism $f : S^\mu \rightarrow P(S^\nu)$,
composing with the map induced by the
canonical inclusion $S^\nu \rightarrow M^\nu$ gives a non-zero homomorphism
$f : S^\mu \rightarrow P(M^\nu)$. Then take a set-tabloid $\{T\}$ with non-zero
coefficient in the image of $e_{t_\mu}$, as before. 
The proof goes through changing columns to rows. Observe that swapping two entries in a row of $\{T\}$ leaves $\{T\}$ unchanged but the permutation  $\rho$ is now odd, so $e_{t_\mu} \rho=-e_{t_\mu}$.
\end{itemize}

The end of the proof goes through essentially unchanged.

\section{Proof of Theorem~\ref{thm:max}}\label{sec:psi}

In this section we prove the following theorem
which determines the minimal constituents of the characters $\psi^{(m^n)}_\nu$ 
defined in Section~\ref{subsec:Q}.

\begin{theorem}\label{thm:psiMin}
Let $\nu$ be a partition of $n$ and $\lambda$ be a partition of $mn$.  Set $\kappa = \nu$ if $m$ is even and $\kappa = \nu'$ if $m$ is odd. Let $k$ be the first part of $\kappa$.
Then $\chi^\lambda$ is a minimal constituent
of $\psi_\nu^{(m^n)}$ if and only if there is a minimal multiset family tuple
$(\mathcal{Q}_1, \ldots, \mathcal{Q}_k)$ of type $\lambda$ such that
each $\mathcal{Q}_j$ has shape $(m^{\kappa_j'})$.
\end{theorem}

Theorem~\ref{thm:max} then follows at once by Equations~\eqref{eqn:Specht_sign_twist} and~\eqref{eqn:sign_twist} 
in Section~\ref{subsec:Q}.

The proof of Theorem~\ref{thm:psiMin} again follows the same structure as that of Theorem~\ref{thm:min}, 
although this time we are usually able to treat the even and odd cases together. We give full details since there
are several places where the change from sets to multisets means that 
new ideas are required.

Let $\nu$, $\kappa$ and $k$ be as in Theorem~\ref{thm:psiMin}. 
Let $(\mathcal{Q}_1, \ldots, \mathcal{Q}_k)$ be a closed
multiset family tuple of type $\lambda$ such that $\mathcal{Q}_j$ has shape $(m^{\kappa'_j})$ for each $j$.
We define the \emph{column multiset-tableau} $T$ and \emph{column multiset-tabloid} $|T|$ corresponding to 
$(\mathcal{Q}_1, \ldots, \mathcal{Q}_k)$ by replacing sets with multisets in Definition~\ref{def:coltab}. Note
that~$T$ and $|T|$ both have shape $\kappa$.
Let $v$ span $\sgn_{S_{mn}}$.
Let 
\[h_{(\mathcal{Q}_1, \ldots, \mathcal{Q}_k)} : \widetilde{M}^\lambda \rightarrow
 \sgn_{S_{mn}} \cotimes P(\widetilde{M}^\kappa)\]
 be the unique $\Q S_{mn}$-homomorphism such that
\[ |t_\lambda| h_{(\mathcal{Q}_1, \ldots, \mathcal{Q}_k)} = (v \otimes |T|) b_{t_\lambda}. \]

\begin{proposition}\label{prop:multisetHoms}
The kernel of $h_{(\mathcal{Q}_1, \ldots, \mathcal{Q}_k)}$ contains every $t_\lambda$-Garnir
element.
\end{proposition}

\begin{proof}
As before, let $1 \le i < \lambda_1$ and let $X = \{i_1, \ldots, i_{\lambda_i'} \}$ 
be the set of entries in column~$i$ of $t_\lambda$.
We have
 \[ |t_\lambda| G_{X\scup \{(i+1)_1\}} h_{(\mathcal{Q}_1, \ldots, \mathcal{Q}_k)}
 = 
\sum_{\tau \in C(t_\lambda)} (v \otimes  |T|\tau)  G_{X \scup \{(i+1)_1\}} \sgn(\tau). \] 
To show the right-hand side is zero, it suffices to 
construct an involution on $C(t_\lambda)$, denoted 
$\tau \mapsto \tau^\star$, with the following two properties:
\begin{itemize}
\item[(a)] if $\tau = \tau^\star$ then $ (v \otimes |T| \tau)  G_{X \scup \{(i+1)_1\}}= 0$, 
\item[(b)] if $\tau \not= \tau^\star$ then $ (v \otimes |T| \tau + v \otimes |T| \tau^*)  G_{X \scup \{(i+1)_1\}}= 0$.
\end{itemize}

 Let $\tau \in C(t_\lambda)$. Consider $|T| \tau$. 
Suppose that $|T|\tau$ has a column with two entries both entirely contained in $X \scup \{ (i+1)_1\}$. Let these entries
be $\{ i_{d(1)} , i_{d(2)}, \ldots, i_{d(m)} \}$ 
and $\{ (i+1)_1 , i_{e(2)}, \ldots, i_{e(m)} \}$. Set
\[ \theta=(i_{d(1)}, (i+1)_1)   ( i_{d(2)}, i_{e(2)}) \cdots (i_{d(m)}, i_{e(m)}) \in S_{X \scup \{ (i+1)_1 \}}.\]
 We have $|T| \tau \theta = - |T| \tau $ since $\theta$ swaps two entries in the same column of~$|T|\tau$.
 Since $v \sgn(\theta) \theta = v$, we get
 \[(v \otimes |T| \tau)(1 + \sgn(\theta) \theta) = v \otimes  |T| \tau + v \otimes |T| \tau\theta =0.\]
Taking coset representatives for 
 $\langle \theta \rangle \le S_{X \scup \{ (i+1)_1 \}}$, it follows 
 that $(v \otimes |T| \tau)G_{X \scup \{(i+1)_1\}} = 0$.
  Hence if we define $\tau^*=\tau$ in this case then (a) holds.
 
 We  now assume that each column of $|T|\tau$ has at most one entry contained in $X \scup \{(i+1)_1\}$.
 Let the entry of $|T|\tau$ containing $(i+1)_1$ be
\begin{align*} 
 B_{(i+1)_1} &= 
 \{  i_{e(1)}, \ldots, i_{e(s)},   (i+1)_1, c(1)_{b(1)}, \ldots, c(m-s-1)_{b(m-s-1)} \}, 
\intertext{where $s\in \N_0$ and $c(1), \ldots, c(m-s-1) \ne i$.
Suppose that $B_{(i+1)_1}$ lies in column $j$ of $|T|\tau$.
This column is defined using the multiset family $\mathcal{Q}_j$. Since~$\mathcal{Q}_j$ is closed, there exist
unique symbols $i_{d(1)}$, \ldots, $i_{d(s)}$, $i_{d(s+1)}$, \hbox{$c(1)_{a(1)}$}, \ldots, $c(m-s-1)_{a(m-s-1)}$ such that the multiset} 
A_{(i+1)_1} &= \{  i_{d(1)}, \ldots, i_{d(s)},   i_{d(s+1)}, c(1)_{a(1)}, \ldots, c(m-s-1)_{a(m-s-1)} \} 
\end{align*}
is also an entry in column $j$ of $|T|\tau$.
Define 
\[ \theta =
(i_{d(1)}, i_{e(1)}) \cdots(i_{d(s)}, i_{e(s)})(i_{d(s+1)}, (i+1)_1 ) \in S_{X \scup \{ (i+1)_1 \}}, \] 
and
\[ \pi = (c(1)_{a(1)}, c(1)_{b(1)}) \cdots(c(m-s-1)_{a(m-s-1)}, c(m-s-1)_{b(m-s-1)}) \in C(t_\lambda). \]
Our assumption ensures that $\pi$ is not the identity. Set $\tau^*=\tau \pi$. 
 Since the column set-tabloids $|T|\tau$ and $|T|\tau^\star$ differ only in indices attached to numbers other than $i$ and $i+1$, we have $\tau^{\star\star} = \tau$.
Since $\pi \theta$ swaps two entries in column~$j$ of $|T|$ we have $|T| \tau \pi \theta= -|T| \tau$.
Hence
\[
\begin{split}
 (v \otimes  |T| \tau+ v \otimes  {}&{} |T| \tau^*) (1+\sgn(\theta) \theta) = \\ &
v \otimes |T| \tau+ v \otimes |T| \tau\theta + v \otimes |T| \tau \pi + v \otimes |T| \tau\pi \theta =0.\end{split} \]
It follows that $ (v \otimes |T| \tau + v \otimes |T| \tau^*)  G_{X \scup \{(i+1)_1\}} = 0$, as required in (b).
\end{proof}

Therefore
$h_{(\mathcal{Q}_1, \ldots, \mathcal{Q}_k)}$ induces a homomorphism
$S^\lambda \rightarrow \sgn_{S_{mn}} \otimes P(\widetilde{M}^\kappa)$,
sending $e_{t_\lambda}$ to $(v \otimes |T|)b_{t_\lambda}$.
Let $\bar{h}_{(\mathcal{Q}_1,\ldots, \mathcal{Q}_k)} : S^\lambda \rightarrow
\sgn_{S_{mn}} \otimes P(S^\kappa)$ denote the composition
of this homomorphism with the canonical quotient map $\sgn_{S_{mn}} \otimes P(\widetilde{M}^\kappa)
\rightarrow \sgn_{S_{mn}} \otimes P(S^\kappa)$. Thus
$e_{t_\lambda} \bar{h}_{(\mathcal{Q}_1,\ldots, \mathcal{Q}_k)} = (v \otimes e_T)b_{t_\lambda}$.

We now obtain the analogues of Lemma~\ref{lemma:evenHoms}, Corollary~\ref{cor:even}
and Proposition~\ref{prop:characterToSetFamily}.

\begin{lemma}\label{lemma:psiHoms}
The homomorphism $\bar{h}_{(\mathcal{Q}_1, \ldots, \mathcal{Q}_k)} : S^\lambda
\rightarrow \sgn_{S_{mn}} \otimes P(S^\kappa)$ is non-zero. 
\end{lemma}

\begin{proof}
The coefficients of $v \otimes \{T\}$ in $(v \otimes e_T)b_{t_\lambda}$
and $(v \otimes \{T\})b_{t_\lambda}$ agree. Since $(v \otimes \{T\})\sigma \sgn(\sigma) =
v \otimes \{T\}\sigma$, this coefficient is the order of the subgroup
of $C(t_\lambda)$ that permutes amongst themselves the indices appearing in each entry of $T$. In particular this coefficient is non-zero.
\end{proof}

%

\begin{corollary}{ \ } \label{cor:phipsi}
\begin{thmlist}
\item If there is a closed  multiset family tuple $(\mathcal{Q}_1,\ldots,\mathcal{Q}_k)$
of type $\lambda$ such that~$\mathcal{Q}_i$ has shape $(m^{\kappa_i'})$ for each $i$,
then $\langle \psi^{(m^n)}_\nu, \chi^\lambda \rangle \ge 1$. 
\item If there is a closed  multiset family tuple $(\mathcal{Q}_1,\ldots,\mathcal{Q}_\ell)$
of type $\lambda$ such that~$\mathcal{Q}_i$ has shape $(m^{\nu_i'})$ for each $i$,
then $\langle \phi^{(m^n)}_\nu, \chi^{\lambda'} \rangle \ge 1$. 
\end{thmlist}
\end{corollary}
  
\begin{proof}
If $m$ is even then, by Equation~\eqref{eqn:QP} in Section~\ref{subsec:Q}, 
the codomain of $\bar{h}_{(\mathcal{Q}_1, \ldots, \mathcal{Q}_k)}$ 
is isomorphic to $Q(S^\nu)$.
If $m$ is odd then the codomain of 
$\bar{h}_{(\mathcal{Q}_1, \ldots, \mathcal{Q}_k)}$ is
$\sgn_{S_{mn}} \otimes P(S^{\nu'})$,
and  by Equations~\eqref{eqn:QP} and~\eqref{eqn:Specht_sign_twist}, we
have isomorphisms
$\sgn_{S_{mn}} \otimes P(S^{\nu'}) 
\cong Q(\sgn_{S_n} \otimes S^{\nu'}) \cong Q\bigl( (S^\nu)^\star \bigr)
\cong Q(S^\nu)$. Since $Q(S^\nu)$ affords the character $\psi_\nu^{(m^n)}$, 
part (i) now follows from Lemma~\ref{lemma:psiHoms}.
Part (ii) then follows
from part (i) using Equation~\eqref{eqn:sign_twist} in Section~\ref{subsec:Q}.
\end{proof}

\begin{proposition}\label{prop:characterToMultisetFamily}
If $\chi^\mu$ is a  constituent of $\psi^{(m^n)}_\nu$ 
then there is a   multiset family tuple  $(\mathcal{R}_1,\ldots,
\mathcal{R}_k)$ of type $\mu$ such
that $\mathcal{R}_j$ has shape $(m^{\kappa_j'})$ for each $j$.
\end{proposition}

\begin{proof}
Arguing as in the proof of Proposition~\ref{prop:characterToSetFamily} if $m$ is even and 
as in Remark~(5) in Section~\ref{sec:oddProof} if $m$ is odd,  
there a non-zero $\Q S_{mn}$-module homomorphism $f:  S^\mu \rightarrow \sgn_{S_{mn}}\otimes P(\widetilde{M}^\kappa)$.  
Let $T$ be a set-tableau such that the coefficient of $v \otimes |T|$ in $e_{t_\mu} f$ is non-zero. 
Suppose that there is a column of $T$ containing entries 
$\{ c(1)_{a(1)}, \ldots, c(m)_{a(m)} \}$ and $\{ c(1)_{b(1)}, \ldots, c(m)_{b(m)} \}$
that are equal up the indices attached to numbers. Let
\[ \rho = (c(1)_{a(1)}, c(1)_{b(1)}) \ldots (c(m)_{a(m)}, c(m)_{b(m)}). \] 
Then $e_{t_\mu}\rho =\sgn(\rho) e_{t_\mu}$, whereas 
\[ (v\otimes|T|)\rho=\sgn(\rho)v \otimes(- |T|)=-\sgn(\rho)(v \otimes |T|), \]
since $\rho$ swaps two entries in a column of $T$. 
 It follows that removing
the indices attached to the numbers in
column $j$ of $|T|$ gives a multiset family of shape  $(m^{\kappa_j'})$. The  multiset family tuple obtained has type $\mu$ since the union of the entries in $|T|$ is $\Omega^\mu$.
\end{proof}

The proof of Theorem~\ref{thm:psiMin} is completed in exactly the same manner as that of Theorem~\ref{thm:min}.

\section{Corollaries}\label{sec:corollaries}

In this section we present a number of corollaries of Theorems~\ref{thm:min} and~\ref{thm:max}.
These include a description of the lexicographically least 
 partitions labelling
an irreducible constituent of $\phi^{(m^n)}_\nu$ or $\psi^{(m^n)}_\nu$, confirming two conjectures of Agaoka \cite{Agaoka}.

\subsection{The conjectures of Agaoka}\label{subsec:Agaoka}
Let $\nu$ be a partition of $n$
and set $\kappa = \nu$ if $m$ is even and $\kappa = \nu'$ if $m$ is odd.
Let $k$ be the first part of $\kappa$.
It follows from Theorem~\ref{thm:min}  that
the lexicographically least 
partition $\lambda$ labelling
an irreducible constituent of $\phi^{(m^n)}_\nu$ is the lexicographically
least type of a  set family tuple $(\mathcal{P}_1, \ldots, \mathcal{P}_k)$ such 
that each $\mathcal{P}_j$ has   shape $(m^{\kappa'_j})$. 
We draw an analogous conclusion from Theorem~\ref{thm:psiMin} regarding $\psi^{(m^n)}_\nu$.
We therefore have the following corollary, which was 
conjectured by Agaoka in \cite[Conjecture~2.1]{Agaoka}.

\begin{corollary}\label{cor:lexleast1}
The lexicographically least  
partition labelling an irreducible
constituent of $\phi^{(m^n)}_\nu$ (respectively $\psi^{(m^n)}_\nu\!$)  
is obtained by taking the join of the lexicographically
least partitions labelling an irreducible constituent of each 
$\phi^{(m^{\kappa'_j})}_{(\kappa'_j)}$ (respectively $\psi^{(m^{\kappa'_j})}_{(\kappa'_j)}\!$).
\end{corollary}

The lexicographically least set families are given by the
colexicographic order on $m$-subsets on $\N$. This order
is defined on distinct $m$-sets $A$ and $B$ by 
$A < B$ if and only if $\max (A \backslash B) < \max (B \backslash A)$.
Given an $m$-subset $B$ of~$\N$, let $B^{\le}$ denote the initial
segment of the colexicographic order ending at~$B$; that is, $B^{\le} =
\{ A \subseteq \N : |A| = m, A \le B \}$.
If $A$ is an $m$-subset of $\N$, and $r$ is minimal
such that  $r \in A$ and $r+1 \not\in A$, then the successor to $A$
in the colexicographic is the set $B = \{1,\ldots, s\} \cup \{r+1\} \cup (A \backslash \{1,\ldots,r\})$
where~$s$ is chosen so that $|B| = m$. Thus the colexicographic
order minimizes the size of the largest element in $B \backslash A$. It
follows that if $B$ is an $m$-subset of $\N$ then $B^{\le}$ is the lexicographically
least set family of its shape. 

An explicit construction of the lexicographically least set family of shape $(m^n)$
follows from the basic results on the colexicographic order 
in \cite[Chapter 5]{Bollobas}.
Pick $p_1$ such that
$\binom{p_1}{m} \le n < \binom{p_1+1}{m}$ and let $\mathcal{P}^{(1)}$ be the set of all
$m$-subsets of $\{1,2,\ldots, p_1\}$. Then, for each $i \in \{2,\ldots,m\}$ 
such that $n > \sum_{j=1}^{i-1} \binom{p_j}{m+1-j}$,
pick $p_i$ such that 
\[ \binom{p_i}{m+1-i} \le n-\sum_{j=1}^{i-1} \binom{p_j}{m+1-j} < \binom{p_{i}+1}{m+1-i}\] 
and let $\mathcal{P}^{(i)}$ be the union of $\mathcal{P}^{(i-1)}$ with the 
set of all sets of the form
$X \cup \{p_{i-1}+1, \ldots p_1+1\}$ 
where $X$ is a $(m+1-i)$-subset of $\{1,2, \ldots p_i\}$. The process terminates
with $p_1>p_2>\cdots >p_r > 0$ 
such that $n= \sum_{j=1}^{r} \binom{p_j}{m+1-j}$. The final set family
$\mathcal{P}^{(r)}$ has shape $(m^n)$. 

The construction of the lexicographically least multiset family of shape $(m^n)$ is entirely analogous.
Let $\multibinom{q}{m}$ denote the number $\binom{q+m-1}{m}$
of multisets of cardinality $m$ with elements taken from $\{1,\ldots,q\}$.
We adapt the above construction and express $n$ as $\sum_{j=1}^{s} \multibinom{q_j}{m+1-j}$
for $q_1\ge q_2 \ge \cdots  \ge q_s > 0$, with weak inequalities since repetitions are allowed.

\begin{corollary}\label{cor:lexleast2}
Set $\kappa = (1^n)$ if $n$ is even and $\kappa = (n)$ if $m$ is odd.
\begin{thmlist}
\item Let $p_1, \ldots, p_r$ be as just defined.
The lexicographically least partition labelling an irreducible constituent
of $\phi^{(m^n)}_\kappa$ is
\[\bigl( (p_1+1)^{a_1}, p_1^{b_1-a_1}, (p_2+1)^{a_2}, p_2^{b_2-a_2}, \ldots, 
(p_{r-1}+1)^{a_{r-1}}, 
p_{r-1}^{b_{r-1}-a_{r-1}}, p_r^{b_r}\bigr),\] 
where $a_i= n - \sum_{j=1}^i \binom{p_j}{m+1-j}$ 
and $b_i= \binom{p_i-1}{m-i}$ for each $i \in \{1,\ldots, r\}$. 
\item Let $q_1, \ldots, q_s$ be as just defined.
The lexicographically least partition labelling an irreducible constituent
of $\psi^{(m^n)}_\kappa$ is
\[\bigl( (q_1+1)^{c_1}, q_1^{d_1-c_1}, (q_2+1)^{c_2}, q_2^{d_2-c_2}, \ldots, 
(q_{s-1}+1)^{c_{s-1}}, q_{s-1}^{d_{s-1}-c_{s-1}}, q_s^{d_s}\bigr),\] 
where $c_i= n - \sum_{j=1}^i \multibinom{q_j}{m+1-j}$ 
and $d_i= \multibinom{q_i+1}{m-i}$ for each $i \in \{1,\ldots, s\}$.
\end{thmlist}
\end{corollary}

\begin{proof} 
Let $\lambda$ be the partition in (i). 
We note that it is possible that $p_i = p_{i+1} + 1$ for one or more indices $i$; in this case
 $b_i-a_i$ may be negative, and
\[ (\ldots ,p_i^{b_i-a_i}, (p_{i+1}+1)^{a_{i+1}}, \ldots ) \] 
should be
interpreted as $(\ldots, p_i^{b_i-a_i+a_{i+1}}, \ldots )$. By Theorem~\ref{thm:min}, it 
is sufficient to prove
that $\lambda$ is the type of the lexicographically least set family
of shape~$(m^n)$, as constructed above. Of course this also shows that $\lambda$
is a well-defined partition.


Let $1 \le x \le p_1+1$.
Note that if $x \le p_j$ then $x$ is contained in exactly~$b_j$ sets in
$\mathcal{P}^{(j)} \backslash \mathcal{P}^{(j-1)}$.
It follows that if $p_{j+1} + 1 < x \le p_{j}$ then $x$ lies
in $b_1 + \cdots + b_{j}$ sets in $\mathcal{P}^{(j)}$ and in
no other sets in $\mathcal{P}^{(r)}$. This is the number of parts of
$\lambda$ not less than~$x$.
If $x = p_{i}+1$ then $x \le p_{i-1}$ and so $x$ lies in $b_1 + \cdots + b_{i-1}$ sets
in $\mathcal{P}^{(i-1)}$ and also in all $a_i$ sets in $\mathcal{P}^{(r)} \backslash
\mathcal{P}^{(i)}$. Hence the total multiplicity of $x$ is $b_1 + \cdots + b_{i-1} + a_i$,
which is again the number of parts of $\lambda$ not less than~$x$.


The proof of (ii) is similar, replacing sets with multisets,
and noting that if $x \in \{1, \ldots, q_1\}$ then the number of multisubsets of
$\{1,\ldots,q_1\}$
 of cardinality $m$ that contain $x$ with multiplicity at least $\ell$ 
is $\multibinom{q_1}{m-\ell}$, and so 
the number of occurrences of $x$ in
all multisubsets of $\{1,\ldots, q_1\}$ of cardinality $m$ is given by 
$\smash{\sum_{\ell=1}^{m}\multibinom{q_1}{m-\ell} = \multibinom{q_1+1}{m-1}}$.
We note that it is possible that $q_{i}=q_{i+1}$ for one or more indices and, in this case, it will be necessary to rearrange the 
parts in the expression given in (ii) to ensure that it is weakly decreasing.
\end{proof}

This result was conjectured by Agaoka in \cite[Conjecture~4.2]{Agaoka}.

Agaoka also conjectured the form of the  lexicographically greatest 
Schur function appearing in $s_\nu \circ s_\mu$ in \cite[Conjecture 1.2]{Agaoka}. This was proved by was Iijima in \cite[Theorem 4.2]{Iijima}. 
Our results provide an alternative proof in the cases $\mu=(n)$ and $(1^n)$. 
Suppose that $\nu$
has exactly $k$ parts and largest part~$\ell$. 
By Theorem~\ref{thm:max}, the lexicographically greatest 
constituent of $\phi^{(m^n)}_\nu$ is $\chi^{((m-1)n+\nu_1, \nu_2, \ldots, \nu_k)}$,  
corresponding to the closed  multiset family tuple with lexicographically greatest 
 conjugate type, namely
$(\mathcal{Q}_1, \ldots, \mathcal{Q}_\ell)$ where
\[ \mathcal{Q}_i=\bigl\{ \{1, \ldots, 1,1\}, \{1, \ldots, 1,2\}, \ldots, \{1, \ldots, 1,\nu_i'\}\bigr\} \] 
for each $i$.
Similarly, by Theorem~\ref{thm:min} and Equation~\eqref{eqn:sign_twist}, 
the lexicographically greatest 
 constituent of $\psi^{(m^n)}_\nu$ is $\chi^{(n^{m-1},\nu_1, \nu_2, \ldots, \nu_k)}$, 
  corresponding to the  set family tuple $(\mathcal{P}_1, \ldots, \mathcal{P}_\ell)$ where 
\[ \mathcal{P}_i=\bigl\{ \{1,2, \ldots, m-1,m\}, \{1,2, \ldots, m-1,m+1\}, \ldots, \{1,2, \ldots, m-1,m+\nu_i'\} \bigr\} \] 
for each $i$.

\subsection{Unique minimal or maximal constituents}
It is natural to ask when $\phi^{(m^n)}_\nu$ has a unique minimal or maximal constituent. This
is easily answered using our results.

\begin{corollary}\label{cor:uniqueminimal}
Let $\nu$ be a partition of $n$. If $m=1$ then $\phi^{(m^n)}_\nu = \chi^\nu$.
If $m > 1$ then $\phi^{(m^n)}_\nu$ has $\chi^\lambda$ as a unique minimal constituent
if and only if 
\begin{itemize}
\item[(i)] $m$ is even, $\nu = (n)$ and $\lambda = (m^n)$;
\item[(ii)] $m$ is even, $\nu = (n-r,r)$ 
and $\lambda = \bigl((m+1)^r,m^{n-2r},(m-1)^r\bigr)$ where $1 \le r \le n/2$;
\item[(iii)] $m$ is odd, $\nu = (1^n)$ and $\lambda = (m^n)$;
\item[(iv)] $m$ is odd, $\nu = (2^r,1^{n-2r})$
and $\lambda = \bigl((m+1)^r,m^{n-2r},(m-1)^r\bigr)$ where $1 \le r \le n/2$.
\end{itemize}
\end{corollary}

\begin{proof}
Suppose that $m > 1$ and $r \ge 3$. 
Let $\mathcal{P}$ be the lexicographically
least 
set family of shape $(m^r)$. Let $X = \{1,\ldots,m-1\}$ and let 
\[ \mathcal{R} = \bigl\{ X \cup \{m\}, X \cup \{m+1\}, \ldots, X \cup \{m+r-1\} \bigr\}. \]
It is easily seen that $\mathcal{R}$ is a 
minimal set family and that $\mathcal{P}$ and $\mathcal{R}$ have different types.
If $r \le 2$ then there is a unique closed set family of shape $(r^2)$.
It now follows from Theorem~\ref{thm:min} that if $m$ is even
then  $\phi^{(m^n)}_\nu$ has a unique minimal constituent, of
the type claimed in (i) and (ii), if and only
if $\nu'_1 \le 2$. The proof is similar when $m$ is odd.
\end{proof}

\begin{corollary}\label{cor:uniqueMaximal}
Let $\nu$ be a partition of $n$.
If $m > 1$ then $\phi^{(m^n)}_\nu$ has
a unique maximal constituent  if and only if $\nu$ has at most two rows. The unique maximal constituent of 
$\phi^{(m^n)}_{(n)}$ is $\chi^{(mn)}$ and the unique maximal constituent of
\smash{$\phi^{(m^n)}_{(n-r,r)}$ is $\chi^{(mn-r,r)}$}.
\end{corollary}
\begin{proof}
Let $\mathcal{P}$ be the  lexicographically
least  
multiset family of shape $(m^r)$, 
let $\mathcal{R} = \bigl\{ \{1,1, \ldots,1,  1\},\{1,1, \ldots,1,  2\}, \ldots, \{1,1, \ldots,1,  r\}
 \bigr\}$,
and argue as in Corollary~\ref{cor:uniqueminimal}, replacing Theorem~\ref{thm:min} with Theorem~\ref{thm:max}.
\end{proof}

\subsection{Further constituents}\label{subsec:further} 
We remark  that since there are closed set families and closed multiset families that are not minimal,
Corollary~\ref{cor:even} is not implied by Theorem~\ref{thm:min} and neither is Corollary~\ref{cor:phipsi} implied
by Theorem~\ref{thm:max}.
For example, let $\mathcal{P}_1$ denote those 2-sets majorized by $\{2,4\}$
 and let $\mathcal{P}_2$ be those  majorized by $\{1,5\}$. Let $\mathcal{R}_1$ be obtained from $\mathcal{P}_1$
 by replacing $\{2,4\}$ with $\{1,5\}$, and let $\mathcal{R}_2$ be obtained from $\mathcal{P}_2$ by replacing
 $\{1,5\}$ with $\{2,3\}$.
  Then the set family tuple $(\mathcal{P}_1 ,  \mathcal{P}_2)$ is closed but not minimal 
  since  $(\mathcal{R}_1, \mathcal{R}_2)$
has strictly smaller type.

\subsection{Rectangular partitions}\label{subsec:rectangular}

As in Section~\ref{subsec:Schur},
let $\Delta^\lambda$ be the Schur functor corresponding to the partition $\lambda$. 
Let $a$, $b \in \N$. By Section~\ref{subsec:Schur},
$\chi^{(a^b)}$ is a constituent of $\phi^{(m^n)}_\nu$
if and only if $\Delta^{(a^b)}(E)$ appears in $\Delta^\nu \bigl( \Sym^m \! E \bigr)$, where
$E$ is a rational vector space of dimension at least~$b$. If~$E$ has dimension exactly $b$
then $\Delta^{(a^b)}(E) \cong (\bigwedge^b E)^{\otimes a}$ and so $\Delta^{(a^b)}(E)$ affords
the polynomial representation $g \mapsto \det(g)^a$ of $\GL(E)$. It follows that
there is a non-zero $\SL(E)$-invariant subspace of $\Delta^\nu \bigl( \Sym^m \! E \bigr)$. 
This observation motivates the following result.

\begin{corollary}\label{cor:rectangular}
Let $a \in \N$ be such that $a \ge m$. 
\begin{thmlist}
\item   If $m$ is odd let $\nu$ denote the partition $(\binom{a}{m}, \ldots, \binom{a}{m})$ where there are exactly
$k$ parts and, if $m$ is even, let $\nu$ denote the conjugate of this partition.
Set $ n = k\binom{a}{m}$ and $b = k\binom{a-1}{m-1}$. 
Then 
\[ \langle \phi^{(m^n)}_\nu, \chi^{(a^b)} \rangle \ge 1. \]
\item Let $\nu$ denote $(\multibinom{a}{m}, \ldots, \multibinom{a}{m})$ where there are exactly $k$ parts. Set $n = k \multibinom{a}{m}$ and $b = k\multibinom{a+1}{m-1}$.  
Then 
\[ \langle \phi^{(m^n)}_\nu, \chi^{(b^a)} \rangle \ge 1. \]
\end{thmlist}
\end{corollary}

\begin{proof}
Consider the  set family tuple $(\mathcal{P}, \ldots, \mathcal{P})$ where $\mathcal{P}$ consists of all $m$-subsets of $\{1,\ldots, a\}$. The shape of $\mathcal{P}$ is $(m^{\binom{a}{m}})$ and the type of the  set family tuple is $(a^{k\binom{a-1}{m-1}})$. Since $\mathcal{P}$ is clearly closed,
the first statement in the corollary now follows from Corollary~\ref{cor:even}, and its analogue for $m$ odd.
Replacing~$\mathcal{P}$ with the set of all multisets of cardinality $m$ with entries taken from $\{1,\ldots, a\}$,
the counting argument in the proof of Corollary~\ref{cor:lexleast2} shows that
we obtain a multiset family tuple of type $(a^{k\multibinom{a+1}{m-1}})$.
 The second statement now
follows similarly from Corollary~\ref{cor:phipsi}(ii).
\end{proof}

For example, $\phi^{(3^8)}_{(4,4)}$ contains $\chi^{(4^6)}$; the corresponding
set family is $\mathcal{P} = \bigl\{ \{1,2,3\}, \{1,2,4\}, \{1,3,4\}, \{2,3,4\} \bigr\}$.
In fact, by Section~\ref{subsec:Agaoka},
every closed  set family tuple or closed multiset family tuple whose type is a rectangular partition
arises from the construction in Corollary~\ref{cor:rectangular}. We note
that in general there are further constituents of $\phi^{(m^n)}_\nu$ 
labelled by rectangular partitions that are not given by this construction.
For example, $\chi^{(4^2)}$ appears in $\phi^{(2^4)}_{(2,2)}$.

\subsection{The decomposition of $\phi^{(2^n)}_{(1^n)}$.}
Let $\theta_n = \phi^{(2^n)}_{(1^n)}$. Remarkably every constituent of $\theta_n$
is both minimal and maximal. We end by proving this as part of the following corollary, which gives
a new proof of the decomposition
of $\phi^{(2^n)}_{(1^n)}$. A notable feature of this proof is
that each constituent is determined by an explicitly defined homomorphism.
For an earlier proof of Corollary~\ref{cor:hook1} using
symmetric functions see \cite[I.~8, Exercise~6(d)]{Macdonald}. 

Given a partition $\alpha$ of $n$ with distinct parts $(\alpha_1,\ldots,\alpha_r)$,
let $2[\alpha]$ denote the partition $\lambda$ of $2n$ such that the leading diagonal
hook-lengths of $\lambda$ are $2\alpha_1,\ldots,2\alpha_r$ and $\lambda_i = 
\alpha_i + i$ for $1 \le i \le r$. 

\begin{corollary}\label{cor:hook1}
For any $n \in \N$ we have
\[ \theta_n = \sum_\alpha \chi^{2[\alpha]} \]
where the sum is over all such partitions $\alpha$ of $n$ with distinct parts.
\end{corollary}

\begin{proof}
By Theorem~\ref{thm:min},
the minimal constituents of $\theta_n$ are given by the types of the minimal set families $\mathcal{P}$ of shape~$(m^n)$.
By Theorem~\ref{thm:max},
the maximal constituents of~$\theta_n$ are given by the conjugates of the types of the minimal multiset 
families~$\mathcal{Q}$ of shape $(m^n)$. The closed set families of shape~$(2^n)$ are
\[\bigcup_{i=1}^r \bigl\{ \{i,i+1\}, \ldots, \{i, i+\alpha_i\} \bigr\}\]
for any $\alpha_1>\alpha_2>\cdots >\alpha_r$ with $\sum_{i=1}^r \alpha_i=n$. Such a set family has type $2[\alpha]$. All such partitions $2[\alpha]$ of $2n$ are incomparable in the dominance order and therefore all label minimal constituents of $\theta_n$. However, $2[\alpha]'$ 
 is the type of the closed multiset family
\[\bigcup_{i=1}^r \bigl\{ \{i,i \}, \ldots, \{i, i+\alpha_i-1\} \bigr\}\]
and hence every minimal constituent is also maximal. We conclude that $\theta_n$ has no further constituents.
\end{proof}

\def\cprime{$'$} \def\Dbar{\leavevmode\lower.6ex\hbox to 0pt{\hskip-.23ex
  \accent"16\hss}D} \def\cprime{$'$}
\providecommand{\bysame}{\leavevmode\hbox to3em{\hrulefill}\thinspace}
\providecommand{\MR}{\relax\ifhmode\unskip\space\fi MR }
\providecommand{\MRhref}[2]{%
  \href{http://www.ams.org/mathscinet-getitem?mr=#1}{#2}
}
\providecommand{\href}[2]{#2}

\end{document}